\documentclass[12 pt]{article}

\usepackage{amsmath,amsthm,amssymb,graphicx,mathrsfs}
\usepackage[hidelinks]{hyperref}
\usepackage{fullpage}
\usepackage{enumerate}
\usepackage{colonequals}
\usepackage{comment}
\usepackage{framed}
\usepackage{subcaption}
\usepackage{float}
\usepackage{stmaryrd}
\usepackage{tikz-cd}

\usepackage[parfill]{parskip}
\makeatletter
\def\thm@space@setup{%
  \thm@preskip=\parskip \thm@postskip=0pt
}
\makeatother

\theoremstyle{plain}
\newtheorem{thm}{Theorem}

\newtheorem{lemma}[thm]{Lemma}
\newtheorem{propn}[thm]{Proposition}
\newtheorem{cor}[thm]{Corollary}

\newtheorem*{Plausible Claim}{Plausible Claim}

\theoremstyle{definition}
\newtheorem{defn}[thm]{Definition}

\newtheoremstyle{example}
{5pt}
{0pt}
{}
{}
{\it}
{.}
{.5em}
{}
\theoremstyle{example}
\newtheorem*{rmk}{Remark}

\newtheorem*{ex}{Example}
\newtheorem*{exs}{Examples}

\newcommand{\R}{\mathbb{R}}

\newcommand{\N}{\mathbb{N}}

\renewcommand{\subset}{\subseteq}

\newcommand{\id}{\mathrm{id}}

\newcommand{\red}[1]{\textcolor{red}{#1}}

\newcommand\restr[2]{{
  \left.\kern-\nulldelimiterspace 
  #1 
  \vphantom{\big|} 
  \right|_{#2} 
  }}

\title{A Dixmier-Malliavin theorem for Lie groupoids}
\author{
Michael Francis\\
Department of Mathematics\\ 
Pennsylvania State University\\ 
University Park, PA 16802, USA\\
\url{mjf5726@psu.edu}
}
\date{September 27, 2020}

\begin{document}

\maketitle

\begin{abstract}
\noindent 
A famous theorem of Dixmier-Malliavin asserts that every  smooth, compactly-supported function on a Lie group   can be expressed as a finite sum in which each term is the convolution, with respect to Haar measure, of two such functions. We establish that the same holds for a Lie groupoid. Most of the heavy lifting is done by a lemma in the original work of Dixmier-Malliavin. We also need the technology of Lie algebroids and the corresponding notion of  exponential map. As an application,  we obtain a result on the arithmetic of  ideals in the smooth convolution algebra of a Lie groupoid arising from functions  vanishing to given order on an invariant submanifold of the unit space.
\end{abstract}

\section{Introduction}

In  a 1960 paper,  Ehrenpreis~\cite{Ehrenpreis} posed a number of related questions including whether every smooth, compactly-supported function $\varphi \in C_c^\infty(\R^n)$ can be ``deconvolved'' as $\varphi=f*g$ where $f,g \in C_c^\infty(\R^n)$. The latter question became known as the \emph{Ehrenpreis factorization problem}. In 1978, Rubel-Squires-Taylor~\cite{RST} showed that  the answer is ``no'' if $n \geq 3$. In the same year, Dixmier-Malliavin~\cite{Dixmier-Malliavin} showed that the answer is still ``no'' if $n=2$. The remaining case $n=1$ was eventually settled in 1999 by Yulmukhametov~\cite{Yulmukhametov}  who showed the answer is ``yes'' for the real line. Also in the positive direction, \cite{Dixmier-Malliavin} gives  the answer to a weaker form of the factorization question to be ``yes'' for any Lie group whatsoever. This is the \emph{Dixmier-Malliavin theorem}.\footnote{More accurately, one of several closely-related Dixmier-Malliavin theorems.}

\begin{thm}[3.1 Th\'{e}or\`{e}me, \cite{Dixmier-Malliavin}]
Let $G$ be a Lie group and form the smooth convolution algebra $C_c^\infty(G)$. Then, every $\varphi \in C_c^\infty(G)$ can be expressed as $f_1 * \psi_1 + \ldots + f_N * \psi_N$ where $f_i,\psi_i \in C_c^\infty(G)$. Moreover, we can choose this factorization such that, for every $i$,  $\mathrm{supp}(\psi_i) \subset \mathrm{supp}(\varphi)$ and  $\mathrm{supp}(f_i) \subset W$, where $W \subset G$ is a neighbourhood of the identity fixed in advance. 
\end{thm} 

The main goal of this paper is to establish the following analogous result for Lie groupoids. 

\begin{thm}\label{main}
Let $G$ be a Lie groupoid and form the smooth convolution algebra $C_c^\infty(G)$. Then, every $\varphi \in C_c^\infty(G)$ can be decomposed as $f_1 * \psi_1 + \ldots + f_N * \psi_N$, where $f_i,\psi_i \in C_c^\infty(G)$. Moreover, we can choose this factorization such that, for every $i$, $\mathrm{supp}(\psi_i) \subset \mathrm{supp}(\varphi)$ and  $\mathrm{supp}(f_i) \subset W$, where $W \subset G$ is a neighbourhood of the unit space of $G$ fixed in advance.
\end{thm} 
 
Note that defining a convolution product on $C_c^\infty(G)$ requires a choice of Haar system but, as different Haar systems lead to canonically isomorphic algebras, issues relating to factorization do not depend on this choice. One could also avoid making a choice entirely by working with appropriate  densities in place of functions. Additionally, the open set $W$ does not  actually need to contain the whole unit space of $G$, only a suitable compact subset. These technicalities, and others, shall be entered into in more detail later in the document. A more comprehensive  statement is given in  Theorem~\ref{mainprecis}.

We furthermore apply our Dixmier-Malliavin theorem to obtain results  on the arithmetic  of certain ideals in  the smooth convolution algebra of a Lie groupoid. Given a closed, invariant submanifold $X$ of the unit space of a  Lie groupoid $G$, one obtains ideals $J_p \subset C_c^\infty(G)$ consisting of the functions which vanish to order $p$ on the restricted groupoid $G_X \subset G$. Our main findings here are:
\begin{align*}
J_\infty * J_\infty = J_\infty&& (J_1)^{*p} = J_p.
\end{align*}
In practical terms, this means that a function vanishing   to infinite order on $G_X$ can be written as a finite sum in which each term is a convolution of two functions vanishing to infinite order on $G_X$, and that a function  vanishing to $p$th order on $G_X$ can be written as  a finite sum in which each term is a $p$-fold convolution of functions which vanish on $G_X$. Note these results on ideals  are only interesting after one has generalized to the groupoid setting. In the group case, the unit space consists of a single point and these  ideals  do not arise at all. 


Elsewhere (\cite{Francis}), we apply the results of this article to investigate the structure of the smooth  convolution algebra $C_c^\infty(G)$ when $G$ is the holonomy groupoid of a singular foliation (\cite{Pradines}, \cite{Pradines-Bigonnet}, \cite{Debord}, \cite{AS}). The holonomy groupoid of  the singular foliation  of $\R$ given by vector fields that vanish to $k$th order  at the origin (\cite{AS}, Example~1.3~(3)) is isomorphic to the transformation groupoid $\R \rtimes_k \R$ of an appropriate flow fixing the origin to $k$th order. We show, by way of an analysis of the ideal structure, that the smooth algebras $C_c^\infty(\R \rtimes_k \R)$ are pairwise nonisomorphic. The C*-completions, on the other hand, fall into two isomorphism classes according to the parity of $k$. This demonstrates that, even in very simple cases,  the smooth algebra of a singular foliation can contain information   that is  washed away after one passes to the  foliation C*-algebra.

We now explain the organization of this article. In Section~2, we use a key lemma of Dixmier-Malliavin to obtain, as a corollary, the preliminary factorization result Theorem~\ref{linefac}. This allows us, given a smooth $\R$-action on a manifold $M$,  to express functions on $M$ as two-term sums in which both terms are the convolution of a function on $\R$ with a function on $M$. Section~3 largely serves a notational  purpose. In it, we lay out our conventions for the Lie algebroid of a Lie groupoid, for Haar measures, and for Lie groupoid actions. Section~4 is quite technical and culminates with  Lemma~\ref{horrid}, giving a framework in which functions on $\R$ can be convolved to yield a function on  a Lie groupoid with appropriate properties. In Section~5, we derive the main result, Theorem~\ref{mainprecis}.  Section~6 prepares the ground for the final section by analyzing the product structure of ideals in the (commutative) algebra of smooth functions on a manifold under pointwise multiplication. This is again somewhat technical, though we do make use  of an interesting inversion principle which connects Schwartz functions to bump functions (Lemma~\ref{corresp}). Finally, in Section~7, we obtain our results on the product structure of  ideals in the smooth convolution algebra of a Lie groupoid. The main finding here is  Theorem~\ref{J2=J}.


\section{Dixmier-Malliavin for $\R$ actions}

Let $X$ be a complete  vector field on a smooth manifold $M$. We denote by $t \mapsto e^{tX}$ the corresponding $1$-parameter group of diffeomorphisms (in other words the flow) that is related to $X$ by
\[ \tfrac{d}{dt}\Big|_{t=0} f(e^{tX}m) = (Xf)(m). \]
This action  of $\R$ on $M$ can alternatively be encoded by a representation $\pi$ of the smooth convolution algebra $C_c^\infty(\R)$ on $C_c^\infty(M)$ which  we call the \emph{integrated form} of the action.  This representation $\pi$ is defined by the following formula:
\begin{align}\label{Rint}
(\pi(f)\psi )(m) = \int_\R f(-t) \psi (e^{tX}m) \ dt. 
\end{align}

The essential ingredient in our Lie groupoid generalization of Dixmier-Malliavin is the following preliminary factorization result:  
\begin{thm}\label{linefac}
Let $\R$ act  smoothly on a manifold $M$ via a complete vector field $X$ and let $\pi$ be the representation of $C_c^\infty(\R)$ on $C_c^\infty(M)$ defined by \eqref{Rint}. Then, for any $\varphi \in C_c^\infty(M)$, there exists $f_0,f_1 \in C_c^\infty(\R)$ and $\psi_0, \psi_1 \in C_c^\infty(M)$ such that 
\begin{align}\label{fact}
\varphi = \pi(f_0)\psi_0+ \pi(f_1) \psi_1.
\end{align} Moreover, this factorization can be taken such that, for $i=0,1$,  $\mathrm{supp}(\psi_i) \subset \mathrm{supp}(\varphi)$ and $\mathrm{supp}(f_i) \subset (-\epsilon, \epsilon)$, where $\epsilon > 0$ is fixed in advance.
\end{thm} 
Although Theorem~\ref{linefac} is not recorded in \cite{Dixmier-Malliavin}, it is a straightforward consequence of the results and techniques therein. The main challenge is to achieve the factorization~\eqref{fact} with smooth $f_i$. It is far easier to achieve the factorization if the $f_i$ are only required to be differentiable of a large, but finite, order (see \cite{Cartier}, pp. 23). The means by which  this simpler result is achieved suggest a strategy for the more difficult result,  so it seems worthwhile to provide a short outline here.
\begin{lemma}\label{Ck}
For every nonnegative integer $k$, there exists $f \in C_c^k(\R)$,  $g \in C^\infty_c(\R)$ such that  
\[ \delta = f^{(k+2)}+g \]
where $\delta$ denotes the delta  distribution at $0$.
\end{lemma}
\begin{proof}
Antidifferentiate the delta distribution $k+2$ times, always picking the antiderivative that vanishes on the negative half line. The result is the $C^k$ function  $F$ vanishing on the negative-half line and satisfying $F(t) = \frac{1}{(k+1)!} t^{k+1}$ for $t \geq  0$. 
Of course, $F^{(k+2)} = \delta$, but $F$ is not compactly-supported. Let $G$ be a $C^\infty$ function that agrees with $F$ outside of a bounded interval. Then,
$\delta = F^{(k+2)} = (F-G)^{(k+2)} + G^{(k+2)} = f^{(k+2)} + g$, 
where $f =F-G \in C_c^k(\R)$ and $g = G^{(k+2)} \in C_c^\infty(\R)$.
\end{proof}
This elementary lemma has a weak  version of Theorem~\ref{linefac} as a corollary. Note   the representation $\pi$ defined by \eqref{Rint} still makes sense for functions that aren't  smooth, but only, say, continuous.  Moreover, $\pi(f) \psi$  still belongs to $C_c^\infty(M)$  when $f \in C_c (\R)$, provided $\psi \in C_c^\infty(M)$. We can even extend the representation $\pi$ to compactly-supported distributions on $\R$, for instance $\pi(\delta) = \mathrm{id}_{C_c^\infty(M)}$. This is mainly a notational point---it is entirely possible to eliminate distributions from the discussion. However, because the corollary below is only being included for motivational reasons, it hardly seems worth it to do so.
\begin{cor}
Let $X$ be a complete vector field on a manifold $M$ and let $\pi$ be the   representation of $C_c(\R)$ on $C_c^\infty(M)$ defined by \eqref{Rint}. Then, for any $\varphi \in C_c^\infty(M)$ and any integer $k \geq 0$, one can write 
\[ \varphi =\pi(f_0) \psi_0 + \pi(f_1) \psi_1 \] 
where $f_0 \in C_c^k(\R)$, $f_1 \in C_c^\infty(\R)$ and $\psi_0,\psi_1 \in C_c^\infty(M)$. 
\end{cor} 
\begin{proof}[Proof sketch]
Applying the above lemma, we can write $\delta = f^{(k+2)} + g$ where $f \in C^k(\R)$, $g \in C_c^\infty(\R)$.  Noting the relation $\pi(h') \psi =\pi(h)X\psi$ (an instance of the integration by parts formula), we get $\varphi = \pi(\delta)\varphi =\pi(f^{(k+2)})\varphi + \pi(g) \varphi = \pi(f)X^{k+2}\varphi + \pi(g) \varphi$ so, setting $f_0 =f$, $f_1=g$, $\psi_0 = X^{k+2} \varphi$, $\psi_1=\varphi$, we are finished.
\end{proof}

A major innovation of \cite{Dixmier-Malliavin} is to achieve approximate representations of $\delta$, analogous to that of Lemma~\ref{Ck}, but in terms of $C^\infty$ functions. Precisely, they prove the following. 

\begin{thm}[2.5 Lemme, \cite{Dixmier-Malliavin}]\label{DistribLemma}
Given any positive sequence $c_k > 0$, there exist functions $f,g \in C_c^\infty(\R)$ and scalars $a_k$ with $|a_k| \leq c_k$ such that 
\[ \delta = g + \sum_{k=1}^\infty a_k f^{(k)}, \] 
where the convergence is in the sense of compactly-supported distributions.\footnote{Explicitly, $h_n \in C_c^\infty(\R)$ converge to $\delta$ in the distributional sense if their supports are uniformly bounded and $\int h_n(x) f(x) \ dx \to f(0)$ for every $f \in C_c^\infty(\R)$. }
\end{thm}

Note that an obvious rescaling argument implies that the functions $f,g$ of the above theorem can be taken to have support contained in $(-\epsilon, \epsilon)$, for any $\epsilon > 0$.

In spite of its simplicity, the methods by which the above statement about 1-dimensional distributions is proved are highly nontrivial.  As Casselman put it in his  exposition \cite{Casselman}, ``Its proof is very intricate, a real tour de force.'' The crucial point turns out to be to find conditions on the growth of a sequence  $\lambda=\{ 0 < \lambda_1 < \lambda_2 < \ldots\}$ that will  guarantee  that the function $\chi_\lambda$ defined as the restriction to $\R$ of the reciprocal of the entire function represented by the infinite product $\prod (1+\frac{z^2}{\lambda_i^2})$ defines a function in the Schwartz space $\mathcal{S}(\mathbb{R})$. Functions of this type are used to prove a Schwartz function analog of Theorem~\ref{DistribLemma} which, in turn, is used to prove Theorem~\ref{DistribLemma}.

 Theorem~\ref{DistribLemma} is exactly the tool needed to give the

\begin{proof}[Proof of Theorem~\ref{linefac}]
Fix $\varphi \in C_c^\infty(M)$. Note that $X\varphi, X^2 \varphi, \ldots$ have supports contained in the support of $\varphi$. We can choose a sequence of positive constants $c_k$ that decays rapidly enough that, for any sequence of scalars $a_k$ with $|a_k| < c_k$, the sum $\sum_k a_k X^k \varphi$ is uniformly convergent to a $C^\infty$ function (this requires a little diagonal selection trick, but is elementary). Obviously, the support of the sum is contained in that of $\varphi$. From Theorem~\ref{DistribLemma}, we can choose scalars $a_k$, $|a_k|<c_k$ and $f,g \in C_c^\infty(\R)$ with supports contained in $(-\epsilon, \epsilon)$ such that $g+\sum_k a_k f^{(k)} \to \delta$, in the distributional sense. Thus, defining $h_n = g+  \sum^n a_k f^{(k)}$, we have that $\pi(h_n) \varphi \to \varphi$ pointwise over $M$. On the other hand, 
\[ 
\pi(h_n) \varphi = \pi(g) \varphi + \sum^n a_k \pi(f^{(k)}) \varphi = \pi(g) \varphi + \pi(f) \sum^n a_k X^k \varphi.
\]
By the choice of constants $c_k$, the series $\sum^n a_k X^k \varphi$ converges uniformly to a function $\psi \in C_c^\infty(\R)$ which furthermore satisfies $\mathrm{supp}(\psi) \subset \mathrm{supp}(\varphi)$. Thus, we arrive at the weak factorization $\varphi = \pi(f) \psi + \pi(g) \varphi$,  which has all the features advertised in the statement.
\end{proof}

\section{Lie groupoid preliminaries}

To a large extent, the purpose of the present  section is to sufficiently fix our notations and conventions that  precise discussion of Lie groupoids is  possible later on.   Throughout,  $G \rightrightarrows B$ denotes a  Lie groupoid with source map $s$ and target map $t$. The inversion map is denoted $\imath : G \to G$ or, frequently, just $\gamma \mapsto \gamma^{-1}$. For peace of mind, it is assumed that  $s$ and $t$ are submersions with $k$-dimensional fibres and that the unit space $B$ is a closed submanifold of $G$. Multiplication is performed from right to left so that $\gamma_1 \gamma_2$ is defined if and only if $s(\gamma_1) = t(\gamma_2)$. We use the (standard) notations $G_x = s^{-1}(x)$ and $G^x  = t^{-1}(x)$ for the source and target fibres.  

\begin{exs} \text{ }
\begin{enumerate}
\item \emph{Transformation groupoid: }Let $H$ be a Lie group acting smoothly on the left of the manifold $B$. Set $G = H \times B$ and define $s,t  : G \to B$ by $s(h,b) = b$ and $t(h,b) = hb$. Define the product by $(h_1, h_2b)(h_2,b)=(h_1h_2,b)$. 
\item \emph{Pair groupoid: }Define $G = B \times B$ and take $s =\mathrm{pr}_1$, $t = \mathrm{pr}_2$, the two coordinate projections. Define the product by $(b_1,b_2)(b_2,b_3)=(b_1,b_3)$. 
\end{enumerate}
\end{exs}

With some superficial differences, our conventions accord with those found in \cite{Mackenzie} and \cite{Renault}. We follow \cite{Mackenzie} in viewing sections of the Lie algebroid of a Lie groupoid as right-invariant vector fields, but depart from \cite{Renault} by working with right Haar measures in place of left Haar measures. This is done so that our measures and our vector fields will both live along the same fibers (namely the source fibres), but is of no real importance; left and right and can always be exchanged using the inversion map.

 \subsection*{The Lie algebroid of a Lie groupoid}

 A vector field on $G$ is said to be  \emph{right-invariant} if it is tangent to the source fibres (i.e. is a section of the distribution $\ker(ds) \subset TG$)  and satisfies
\begin{align*} 
(R_\gamma)_* X = X && \gamma \in G,
\end{align*}
where $R_\gamma$ denotes right-multiplication by $\gamma$. The equation above  is a bit imprecise because $R_\gamma$ is not defined on all of $G$, but instead is a diffeomorphism $G_{t(\gamma)} \to G_{s(\gamma)}$. A more accurate formulation would be
\begin{align*}
 (R_\gamma)_* \left( \restr{X}{G_y} \right) = \restr{X}{G_x} && \gamma \in G_x \cap G^y,
 \end{align*}
where the restricted vector fields make sense because $X$ was assumed  tangent to the source fibres. Because we can write any $\gamma \in G$ as $R_\gamma( t(\gamma))$, it follows that a right-invariant vector field is completely determined by its restriction to $B$.

\begin{defn}[Definition 3.1, \cite{Mackenzie}]
As a vector bundle over the base manifold $B$, the  \emph{Lie algebroid} $AG$ of $G$ is the restriction of the source fibre tangent bundle $\ker(ds) \subset TG$ to $B$.
\end{defn}

Every section of $AG$ extends uniquely to a right-invariant vector field   on $G$.  Thus, right-invariant vector fields on $G$ are in 1-1 correspondence with sections of the Lie algebroid $AG$. This is the Lie groupoid counterpart of the analogous pair of descriptions for the Lie algebra of a Lie group. We shall tend to abuse notation, denoting a right-invariant vector and its restriction to $B$ (a section of $AG$) by the same symbol.

\begin{ex}
If $G$ is the pair groupoid $B \times B$, then the source fibres are just the   slices $B \times \{b\}$, $b \in B$. A right-invariant vector field on $G$ amounts to a single vector field $X$ on $B$ copied on each slice. Meanwhile, the Lie algebroid obviously identifies with $TB$, so sections of the Lie algebroid also correspond in an obvious way  to  vector fields on $B$. 
\end{ex}

There is a bundle map $AG \to TB$ called the \emph{anchor map} defined simply by restricting the differential $dt : TG \to TB$ of the target submersion to $AG$. It is common practice to denote the anchor map by  $\# : AG \to TB$, but we shall avoid this notation because it overlaps with  common notation for  fundamental vector fields in the context of a Lie group actions, and we will be considering groupoids acting on manifolds. Instead, given $X \in C^\infty(B,AG) = \{ \text{right-invariant vector fields on G}\}$, we write $X^B$ for the corresponding vector field on $B$. Thus, 
\begin{align}\label{X^B}
X^B(b) = \tfrac{d}{dt} t(e^{tX} b) \Big|_{t=0}.
\end{align}
The vector fields $X$ and $X^B$ are $t$-related; if $X$ is a complete, right-invariant vector field on $G$, then $X^B$ is a complete vector field on $B$ and the target submersion $t: G \to B$ is equivariant for the $\R$-actions on $G$ and $B$.

An irritation that does not arise in the Lie group context, but does for Lie groupoids, is the potential for  right-invariant vector fields to have incomplete flows. For example, in the case of the pair groupoid $G = B \times B$, the flow of a right-invariant vector field is just the flow of an arbitrary vector field on $B$, copied on each slice $B \times \{b\}$. The following proposition, however, at least shows that complete,  right-invariant vector fields are in plentiful supply.

\begin{propn}\label{compactcomplete}
Let $G \rightrightarrows B$ be a Lie groupoid, $X$   a compactly-supported section  of the Lie algebroid $AG$. Then,  $X$, considered as a right-invariant vector field\footnote{Possibly no longer compactly-supported; consider what happens if $G$ is a noncompact Lie group.} on $G$, is complete.
\end{propn}
\begin{proof}
Let $\phi : W \to G$ be the (maximal) flow of $X$. So,  $W$ is an open subset of $\R \times G$ containing $\{0\} \times G$. To see $\phi$ is complete, it suffices to show $(-\epsilon,\epsilon) \times G \subset W$ for some $\epsilon >0$. Let $K \subset B$ be a compact set containing $\{b \in B : X(b) \neq 0\}$. By an easy compactness argument, there exists an $\epsilon > 0$ such that $(- \epsilon, \epsilon) \times K \subset W$. In fact, since $X$ vanishes on $B -K$, we actually have $(-\epsilon,\epsilon) \times B \subset W$. But then, for any $\gamma \in G$, we have, using the right-invariance,  the integral curve $t \mapsto \phi_t( t(\gamma))\gamma  : (-\epsilon, \epsilon) \to G$ passing through $\gamma$ at $t=0$, and so $(-\epsilon,\epsilon) \times G \subset W$, as was to be proven.
\end{proof} 

Note that  right-invariance of a complete vector field $X$ on $G$ can also be formulated as a property of its flow; each diffeomorphism $e^{tX}$ should be right-invariant in the sense that it preserves the $s$-fibers and satisfies $e^{tX}(\gamma_1 \gamma_2) = (e^{tX}\gamma_1) \gamma_2$ whenever $\gamma_1 \gamma_2$ is defined.

\subsection*{Haar systems and convolution}

Defining a convolution product on $C_c^\infty(G)$ for a Lie groupoid $G \rightrightarrows B$ requires a choice of (smooth) Haar system. This choice is ultimately unimportant; the algebras associated to different Haar measures are canonically isomorphic. Indeed, by replacing functions on $G$ with sections of an appropriate density bundle,  it is possible to obtain the convolution algebra in a fully intrinsic way which does not require a choice of Haar system. See the discussion following Definition~2 in \cite{Connes},~Section~2.5. In this article, we will stick with functions, however, in large part to better resemble the classical Dixmier-Malliavan theorem.

For a given a submersion $p : W \to B$, say with $k$-dimensional fibres, a  (smooth) \emph{fibrewise measure} is a collection $\lambda = (\lambda_b)$ of smooth measures   on the fibres  $p^{-1}(b)$, $b \in B$ such that, in any open subset of $W$ small enough to be identified with  $\R^k \times U$ for  $U$ an open set in $B$ in such a way that   $p$ is identified with the factor projection $\R^k \times U \to U$, the measures take  the form $\lambda_b = \rho(\cdot, b)  dt_1 \cdots dt_k$, $b \in U$, where $\rho$ is a smooth, positive-valued function on $\R^k \times U$ and $dt_1 \cdots dt_k$ is the standard volume measure copied on each fibre $\{b\} \times \R^k$. Equivalently, one can think of $\lambda$ as a globally positive section of the density bundle of the $p$-vertical subbundle $\ker(dp) \subset TW$. We list a few basic properties:
\begin{enumerate}
\item Fibrewise measures are unique up to rescaling (corresponding to the fact that the density bundle is trivial and oriented); if $\lambda$ and $\nu$ are fibrewise measures for $p:W \to B$, then there is a unique positive-valued, smooth function $\rho$ on $W$ such that $\nu = \rho \lambda$.  
\item A fibrewise measure $\lambda$ for $p: W \to B$  determines a corresponding ``integration along fibres'' map $p_! :  C_c^\infty(W) \to C_c^\infty(B)$.
\item Fibrewise measures  can be pulled back. Suppose, $p:W \to B$ is a submersion and $\mu: M \to B$ is a smooth map so that $\mathrm{pr}_2 : W \times_{p,\mu} M \to M$ is a submersion.\footnote{The fibre product $W \times_{p,\mu} M = \{ (\gamma,m)  : p(\gamma)=\mu(m)\}$ is a closed submanifold of $W \times M$. This follows from writing it as the preimage of the diagonal $\Delta \subset M \times M$ under the map $p \times \mu : G \times M \to M \times M$ and noting that, because $p$ is a submersion, the latter map is transverse to $\Delta$.}
\[ \begin{tikzcd} 
W \times_{p,\mu} M \ar[d,"{\mathrm{pr}_2}"] \ar[r,"{\mathrm{pr}_1}"]& W \ar[d,"p"]\\
M \ar[r,"\mu"]& B 
\end{tikzcd} \]
Then, a fibrewise measure $\lambda$ for $p$ determines a fibrewise measure $\lambda_M$ for $\mathrm{pr}_2$ by  identifying $\mathrm{pr}_2^{-1}(m)$ with $p^{-1}(\mu(m))$ in the obvious way. If preferable, one may obtain this pullback construction in two stages, expressing it in terms of appropriate  product and restriction constructions.
\end{enumerate}

A \emph{(smooth, right) Haar system} $\lambda$ for a Lie groupoid $G \rightrightarrows B$ is a  fibrewise measure $\lambda = (\lambda_b)_{b \in B}$ for the source submersion $s : G \to B$ that is right-invariant in the sense that, for any $\gamma \in G$, the right-multiplication $R_\gamma$ is a measure isomorphism from  $(G_{t(\gamma)}, \lambda_{t(\gamma)})$ to $(G_{s(\gamma)},\lambda_{s(\gamma)})$.   

Recall the Haar measure of a Lie group can be constructed by ``bare hands'' by choosing a positive density on the tangent space of the identity and  then using translation operations to trivialize the tangent bundle and obtain a corresponding  globally-positive density on the whole group, which is translation-invariant by construction. The construction of a Haar system for a Lie groupoid $G \rightrightarrows B$ proceeds along analogous lines: for any $\gamma \in G$, the right-translation $R_\gamma : G_{t(\gamma)} \to G_{s(\gamma)}$ sends $t(\gamma) \mapsto  \gamma$. Thus, any  globally-positive section of the density bundle of $AG \to B$ can be canonically extended to a globally-positive section of the density bundle of $\ker(ds) \subset TG$, the subbundle of tangent spaces to the source fibres.  By construction, this extension is right-invariant in an obvious sense. Along the same lines, one sees that the (smooth) Haar measure of a Lie groupoid is unique up to multiplication (appropriately defined) by a smooth, positive-valued function on the base manifold.

Once a Haar measure $\lambda$ has been fixed for $G$, we obtain a convolution operation  $*$ with respect to which $C_c^\infty(G)$ becomes a (generally noncommutative) algebra.
\begin{align}\label{convdef}
(f*g)(\gamma_0) = \int_{G_{t(\gamma_0)}} f( \gamma^{-1}) g (\gamma \gamma_0) \ d\lambda_{t(\gamma_0)}
\end{align}

\subsection*{Lie groupoid actions}

Let $G \rightrightarrows B$ be a Lie groupoid. Let $M$ be a manifold with a given smooth map $\mu : M \to B$ called the \emph{momentum map}.   By definition, a  \emph{left action} of $G$ on $M$ is a smooth product $G\times_{s,\mu}M \ni (\gamma, m)  \mapsto \gamma \cdot m \in  M$, such that $\mu(\gamma \cdot m) = t(\gamma)$ for all $(\gamma, m) \in G \times_{s,\mu} M$, $\mu(m) \cdot m = m$ for all $m \in M$ and $(\gamma_1\gamma_2)\cdot m = \gamma_1 \cdot (\gamma_2 \cdot m)$ for all $\gamma_1,\gamma_2 \in G$, $m \in M$ with $s(\gamma_1) = t(\gamma_2)$, $s(\gamma_2) = \mu(m)$.

\begin{exs}
\text{ }
\begin{enumerate}
\item A Lie group $H$ acting smoothly on the left of a manifold $M$ can be considered as a Lie groupoid action by taking $G=H$, $B=\{1_G\}$ and $s,t,\mu$ to be the collapsing maps onto the one-point space $B$. 
\item Every Lie groupoid $G \rightrightarrows B$ acts on its own arrow space $G$  by taking the momentum map  to be the target submersion $t:G \to B$ and the action map to be the groupoid multiplication. 
\item Every Lie groupoid $G \rightrightarrows B$ acts on its own unit space $B$ by taking the momentum map to be the identity map on $B$. The fibre product $G \times_{s,\mathrm{id}} B$ canonically identifies with $G$ via, $\gamma \mapsto (\gamma, s(\gamma))$ and, under this identification, the action map is just the target submersion $t: G \to B$. 
\end{enumerate}
\end{exs}

Recall that, when a Lie group $H$ acts on a manifold $M$, each Lie algebra element $X \in \mathfrak{h}$ determines a corresponding \emph{fundamental vector field}  $X^\#$ on $M$. Similarly, when a Lie groupoid $G$ acts on a manifold $M$ with momentum map $\mu$, each Lie algebroid  section $X \in C_c^\infty(B,AG)$ determines a vector field $X^M$ on a $M$ by the formula:
\begin{align}\label{X^M} X^M(m) = \tfrac{d}{dt}  \Big[ e^{tX} \mu(m) \cdot m \Big]_{t=0}. \end{align}
The vector field $X^M$ is complete if $X$ is and satisfies the following right-invariance condition:  $e^{tX^M}( \gamma \cdot m)=(e^{tX} \gamma ) \cdot m$  whenever $\gamma \cdot m$ is defined.

If, additionally, a Haar system $\lambda$ has been specified for $G$, then an  action of $G$ on $M$ determines a representation $\pi$ of the convolution algebra $C_c^\infty(G) = C_c^\infty(G,\lambda)$ on $C_c^\infty(M)$ called the \emph{integrated form of the action} according to the following formula:
\begin{align}\label{wah}
(\pi(f) \psi )(m)  =\int_{G_{\mu(m)}} f(\gamma^{-1}) \psi( \gamma \cdot m) \ d\lambda_{\mu(m)}.
\end{align}
Take note that, in the special case when $G$ is acting on itself from the left, we recover \eqref{convdef}, the convolution product on $C_c^\infty(G)$.

The action of $G$ on $M$ can be packaged as a  Lie groupoid $G \ltimes M \rightrightarrows M$ called the \emph{transformation groupoid} of the action. This is done by taking $G \ltimes M = G \times_{s,\mu} M$ with structure maps defined as follows:
\begin{align*}
\text{source $\sigma$ :} && (\gamma,m) \mapsto m \\
\text{target $\tau$ :}&& (\gamma,m) \mapsto \gamma \cdot m \\
\text{inversion $\jmath$ :} && (\gamma,m) \mapsto (\gamma^{-1}, \gamma \cdot m) \\
\text{multiplication :}  && (\gamma_2,\gamma_1 \cdot m)( \gamma_1, m) = (\gamma_1\gamma_2,m)
\end{align*}
Note that the relation $\tau = \sigma \circ \jmath$ shows that the action map is in fact a submersion. The  Haar measure $\lambda$  on $G$ determines  a  corresponding Haar system $\lambda_M$ on $G \ltimes M$, using the obvious identification of each $(G \ltimes M)_m = G_{\mu(m)} \times \{m\}$ with $G_{\mu(m)}$.

\section{Relating $\R$ actions to groupoid actions}

This section is devoted to proving the somewhat technical Lemma~\ref{horrid} below.  Accordingly, most of the notations set down below can safely be forgotten once this end has been achieved. As always, $G \rightrightarrows B$ is a Lie groupoid with source $s$, target $t$ and inversion map $\imath$ and  Haar system $\lambda$. Let  $M$ be  a $G$-space with momentum map $\mu : M \to B$.  Let  $\pi$ be the corresponding representation of $C_c^\infty(G)$ on $C_c^\infty(M)$ defined by \eqref{wah}.
\[(\pi(f) \psi )(m)  =\int_{G_{\mu(m)}} f(\gamma^{-1}) \psi( \gamma \cdot m) \ d\lambda_{\mu(m)} \]
Let  $X_1,\ldots,X_k \in C_c^\infty(B,AG)$, thought of as complete, right-invariant vector fields on $G$. Correspondingly, we have complete vector fields $X_1^B,\ldots, X_k^B$ on $B$ and and  $X_1^M,\ldots,X_k^M$ on $M$ defined by \eqref{X^B} and \eqref{X^M}, respectively. Let $\pi_1,\ldots, \pi_k$ be the representations of $C_c^\infty(\R)$ on $C_c^\infty(M)$ associated to the complete vector fields $X_1^M,\ldots,X_k^M$ in accordance with \eqref{Rint}.
\begin{align*} (\pi_i(f) \psi )(m) = \int_\R f(-t) \psi(e^{t X_i^M} m) \ dt \end{align*} 
Our basic goal is to work out the relationship between $\pi$ and $\pi_1,\ldots,\pi_k$ in neighbourhoods of $G$ that are parametrized by the map  $u : \R^k \times B \to G$ defined by
\begin{align*}
u(t_1,\ldots,t_k,b) = e^{t_1X_1} \cdots e^{t_k X_k}b.
\end{align*}
We find it helpful to introduce the following operation, as an intermediary between $\pi$ and the $\pi_i$. Given $f \in C_c^\infty(\R^k\times B)$, we define a convolution operation $\widetilde \pi (f)$ on $C_c^\infty(M)$ by 
\begin{align*}
(\widetilde \pi(f)\psi )(m) = \idotsint_{\R^k}  f(-t_k,\cdots,-t_1, \mu( e^{t_1X_1^M} \cdots e^{t_kX_k^M}m)) \psi (e^{t_1X_1^M}\cdots e^{t_k X_k^M}m) \ dt_1 \cdots dt_k.
\end{align*}
The reason for using precisely the above expression will hopefully be made clear shortly. For now, let us note that the following relationship between $\pi_1,\ldots,\pi_k$ and $\widetilde \pi$ is a trivial  consequence of the definitions. 
\begin{lemma}\label{easierlemma}
Suppose $f_1,\ldots,f_k \in C_c^\infty(\R)$, $\chi \in C_c^\infty(B)$ and define $f \in C_c^\infty(\R^k \times B)$ by $f= f_k \otimes \cdots \otimes f_1 \otimes \chi $.  Then,   
\begin{align*}
\widetilde \pi (f) \psi = \pi_k(f_k) \cdots \pi_1(f_1) \psi
\end{align*}
holds whenever $\psi \in C_c^\infty(M)$ has $\chi \equiv 1$ on $\mu(\mathrm{supp}(\psi))$.
\end{lemma}

Next, we relate $\widetilde \pi$ to $\pi$. 

\begin{lemma}\label{pi-pitilde}
Suppose that $W$ is an open subset of $\R^k \times B$ that is mapped diffeomorphically by $u$ onto an open subset of $G$.  Then, there exists a linear bijection $\theta_W$ from  $C_c^\infty(W) \subset C_c^\infty(\R^k \times B)$ to $C_c^\infty(u(W)) \subset C_c^\infty(G)$ such that  
\begin{align*}
\widetilde \pi( f) \psi = \pi( \theta_W(f)) \psi  
\end{align*}
holds for all  $f \in C_c^\infty(W)$ and $\psi \in C_c^\infty(M)$. 
\end{lemma}

We remark that the bijection $\theta_W$ above is independent of the manifold  $M$ and the given action of  $G$; the same $\theta_W$ works for all $G$-sets. The basic idea is to define $\theta_W(f)$  as the pushforward of $f$ by $u$, followed by multiplication by a suitable Jacobian factor.

Before proceeding to the proof of Lemma~\ref{pi-pitilde} we find it useful to re-express the representations $\pi$ and $\widetilde \pi$ in a somewhat more abstract form. Form the transformation groupoid $G \ltimes M \rightrightarrows M$ with source $\sigma$, target $\tau$, inversion map $\jmath$ and induced Haar system $\lambda_M$.  Given $f \in C_c^\infty(G)$ and $\psi \in C_c^\infty(M)$, we define $f \times \psi \in C_c^\infty(G \ltimes M)$ by restricting $f \otimes \psi$ to $G \ltimes M \subset G \times M$
\begin{align*}
(f \times \psi)(\gamma,m) = f(\gamma) \psi(m).
\end{align*}
We can express $\pi$ in terms of the above operations as follows:
\begin{align}\label{abspi}
\pi(f) \psi = \sigma_!( (f \times \psi) \circ \jmath) && f \in C_c^\infty(G), \psi \in C_c^\infty(M),
\end{align}
where $\sigma_! : C_c^\infty(G \ltimes M) \to C_c^\infty(M)$ is the integration along fibres map associated to $\lambda_M$. 

Next, we find an analogous expression for $\widetilde \pi$. First, define the following maps:
\begin{align*}
u &: \R^k \times B \to G && (t_1,\ldots, t_k,b) \mapsto e^{t_1X_1} \cdots e^{t_kX_k} b  \\
\widetilde s &: \R^k \times B \to B && (t_1,\ldots, t_k,b)  \mapsto b \\
\widetilde \imath &: \R^k \times B \to \R^k \times B && (t_1,\ldots, t_k, b) \mapsto (-t_k,\ldots,-t_1,e^{t_1X_1^B} \cdots e^{t_kX_k^B} b) \\
&&&  \\
v &: \R^k \times M \to G \ltimes M&& (t_1,\ldots, t_k,m) \mapsto (e^{t_1X_1} \cdots e^{t_kX_k} \mu(m),m) \\
\widetilde \sigma &: \R^k \times M \to M && (t_1,\ldots, t_k,m)  \mapsto m  \\
\widetilde \jmath &: \R^k \times M \to \R^k \times M && (t_1,\ldots, t_k, m) \mapsto (-t_k,\ldots,-t_1,e^{t_1X_1^M} \cdots e^{t_kX_k^M} m).
\end{align*}
We give $\widetilde s$ and $\widetilde \sigma$ the obvious fibrewise measures, copying the standard volume measure of $\R^k$ on each fibre, and denote the associated integration along fibre maps by $\widetilde s_!$ and $\widetilde \sigma_!$. Notice that $\widetilde \imath$ and $\widetilde \jmath$ are order-2 diffeomorphisms and that the following intertwining relations are satisfied.
\begin{align*}
u \circ \widetilde \imath = \imath \circ u && \widetilde  s = s \circ u && v  \circ \widetilde \jmath = \jmath \circ v && \widetilde \sigma = \sigma \circ v 
\end{align*}
Given $f \in C_c^\infty(\R^k \times B)$ and $\psi \in C_c^\infty(M)$, we define $f \times \psi \in C_c^\infty(\R^k \times M)$  by
\[ (f \times \psi) (t_1,\ldots,t_k,m) = f(t_1,\ldots,t_k,\mu(m)) \psi(m). \]
By analogy with \eqref{abspi}, we give the following expression for $\widetilde \pi$ in terms of the above operations.
\begin{align}\label{abspitild}
\widetilde \pi(f) \psi  = \widetilde \sigma_!( (f \times \psi) \circ \widetilde \jmath) && f \in C_c^\infty(\R^k \times B), \psi \in C_c^\infty(M)
\end{align}
With these preparations and notations, we proceed to the

\begin{proof}[Proof of Lemma~\ref{pi-pitilde}]
The relation  $u\widetilde \imath=\imath u$ implies that $W' \colonequals \widetilde \imath(W)$ is also mapped diffeomorphically by $u$ onto an open subset of $G$.  Thus, there exists a smooth, positive-valued function $\rho_{W'}$ on $W'$ such that the pullback of the Haar measure along $u$ to $W' \subset \R^k \times B$ equals $\rho_{W'} \ dt_1 \cdots dt_k$. Thus, for all $f \in C_c^\infty(u(W'))$, we have
\[ \widetilde s_!( (f \circ u|_{W'}) \rho_{W'}) = s_!(f). \]
Let $\Omega = ( \id \times \mu)^{-1}(W)$, an open subset of $\R^k \times M$, and note that $v$ maps $\Omega$ diffeomorphically onto an open subset of $G \ltimes M$. The relation $v \widetilde \jmath=\jmath v$ implies that  $\Omega' \colonequals \widetilde \jmath(\Omega)$ is also mapped diffeomorphically by $v$  onto an open subset of $G \ltimes M$.  The pullback of the Haar measure $\lambda_M$ of $G \rtimes M$ along $v|_{\Omega'}$ is  $\rho_{\Omega'} \ dt_1 \ldots dt_k$, where $\rho_{\Omega'} = \rho_{W'}(\id \times \mu)$. Thus, for all $f \in C_c^\infty(v(\Omega'))$, we have
\begin{align}\label{ohdear} 
\widetilde \sigma_!( (f \circ v|_{\Omega'}) \rho_{\Omega'}) = \sigma_!(f). 
\end{align}
Take $\theta_W$ to be the bijection $C_c^\infty(W) \to C_c(u(W))$  determined by 
\begin{align*} 
(\theta_W(f) \circ u|_W)( \rho_{W'} \circ \widetilde \imath ) = f   && f \in C_c^\infty(W).
\end{align*}
 Then, for any $f \in C_c^\infty(W)$ and $\psi \in C_c^\infty(M)$, a simple calculation shows that
\[  
f \times \psi 
= (( \theta_W(f) \circ u|_W)(\rho_{W'} \circ \widetilde \imath)) \times \psi 
= ((\theta_W(f) \times \psi) \circ v|_{\Omega})( \rho_{\Omega'} \circ \widetilde \jmath) 
\]
and so
\[ (f \times \psi) \circ  \widetilde \jmath =  (( \theta_W(f) \times \psi) \circ  \jmath  \circ v|_{\Omega'}) \rho_{\Omega'}. \]
Thus, applying \eqref{abspi}, \eqref{abspitild} and \eqref{ohdear}, we find that
\[ \widetilde \pi(f)\psi = \widetilde \sigma_!((f \times \psi) \circ  \widetilde \jmath) =  \sigma_!(( \theta_W(f) \times \psi) \circ \jmath ) = \pi(f) \psi, \]
as was to be proven.
\end{proof}

Combining Lemma~\ref{pi-pitilde} and Lemma~\ref{easierlemma}, we obtain the following technical lemma below, which has been our goal throughout the section.

\begin{lemma}\label{horrid}
Let $G \rightrightarrows B$ be a Lie groupoid with a given Haar system.  Let $K$ be a compact subset of $B$ and let $W$ be an open subset of $G$ containing $K$. Suppose $X_1,\ldots,X_k \in C_c^\infty(B,AG)$  constitute a base for $AG$ over each point of $K$. Then, there exists an $\epsilon > 0$ such that, for any $f_1,\ldots,f_k \in C_c^\infty(\R)$ having $\mathrm{supp}(f_i) \subset (- \epsilon,\epsilon)$, there exists an $f \in C_c^\infty(G)$ with $\mathrm{supp}(f) \subset W$ with the property that, for any $G$-space $M$ with momentum map $\mu$   and any $\psi \in C_c^\infty(M)$ with  $\mu(\mathrm{supp}(\psi)) \subset K$, one has
\[ \pi(f) \psi = \pi_1(f_1) \cdots \pi_k(f_k) \psi \]
where  $\pi_1,\ldots,\pi_k$ are the representations of $C_c^\infty(\R)$ on $C_c^\infty(M)$ associated by  \eqref{Rint} to the corresponding complete vector fields $X_1^M, \cdots, X_k^M$  on $M$ and $\pi$ denotes the representation of $C_c^\infty(G)$ on $C_c^\infty(M)$ given by \eqref{wah}.
\end{lemma}

\begin{proof}
Let $u : \R^k \times M \to G$ be defined by $u(t_1,\ldots,t_k,b) = e^{t_1X_1} \ldots e^{t_kX_k}b$. By an inverse function theorem/compactness argument, there exists an open subset $U$ of $B$ containing $K$ and an $\epsilon > 0$ such that $u$ maps $(-\epsilon,\epsilon)^k \times U$ diffeomorphically onto an open subset of $G$. By possibly shrinking $\epsilon$ and $U$, we may assume that $u( (-\epsilon, \epsilon)^k \times U) \subset W$.  Let $\chi  \in C_c^\infty(B)$ satisfy  $\mathrm{supp}(\chi) \subset U$ and $\chi \equiv 1$ on $K$. Let $f_1, \ldots, f_k \in C_c^\infty(R)$ have $\mathrm{supp}(f_i) \subset (-\epsilon, \epsilon)$ so that $f = f_k \otimes \cdots \otimes f_1 \otimes \chi \in C_c^\infty(\R^k \times B)$ has $\mathrm{supp}(f) \subset (-\epsilon,\epsilon)^k \times U$. Then, for any $\psi \in C_c^\infty(M)$ with $\mathrm{supp}(\psi) \subset K$, we have
\[ \pi_k(f_k) \cdots \pi_1(f_1) \psi = \widetilde \pi( f) \psi = \pi( \theta_W(f)) \psi \]
where the first equality comes from Lemma~\ref{easierlemma} and the second from Lemma~\ref{pi-pitilde}. Since $\theta_W(f)$ is supported in $u( (-\epsilon,\epsilon)^k \times U) \subset W$, we are finished.
\end{proof}

\section{Proof of main theorem}

We have made the necessary preparations to prove Theorem~\ref{main}, stated in the introduction. In fact, we can prove the following more general result.

\begin{thm}
Let $G \rightrightarrows B$ be a Lie groupoid with a given   Haar system, let $M$ be a $G$-space with momentum map $\mu$ and let $\pi$ be the corresponding representation of  $C_c^\infty(G) = C_c^\infty(G,\lambda)$ on $C_c^\infty(M)$, i.e. the integrated form of the action defined by \eqref{wah}.   Then, for every $\varphi \in C_c^\infty(M)$, there exist $f_1,\ldots,f_N \in C_c^\infty(G)$ and $\psi_1, \ldots, \psi_N \in C_c^\infty(M)$ such that 
\[ \varphi = \pi(f_1)\psi_1 + \ldots + \pi(f_N) \psi_N. \] 
Moreover, this factorization can be taken such that, for all $i$, $\mathrm{supp}(\psi_i) \subset \mathrm{supp}(\varphi)$ and   $\mathrm{supp}(f_i) \subset W$, where $W$ is a prescribed open subset of $G$ containing $\mu( \mathrm{supp}(\varphi))$.
\end{thm}

\begin{proof}
Let $K = \mu(\mathrm{supp}(\varphi))$ and fix an open set $W$ in $G$ containing $K$. It is enough to prove the theorem under the additional hypothesis that the Lie algebroid $AG$ is trivial (as a bundle) over a neighbourhood of $K$. Indeed, we can use a partition of unity on $B$ to write any $\varphi$ as a finite sum of functions that satisfy this extra hypothesis and have support contained in that of the original $\varphi$. If each summand can be decomposed in the desired way, then so can the sum, simply by adding up the decompositions.

Under this extra assumption, there exist sections $X_1,\ldots, X_k \in C_c^\infty(B,AG)$ that constitute a frame of $AG$ over each point in $K$. Let $X_1^M,\ldots, X_k^M$ denote the corresponding complete vector fields on $M$ and let $\pi_1,\ldots, \pi_k$ denote the corresponding representations of $C_c^\infty(\R)$. Let $\epsilon > 0$ come from Lemma~\ref{horrid}.

Applying Theorem~\ref{linefac} with $X=X_k^M$ and $\psi = \varphi$, we can write
\[ \varphi = \pi_k(f_0)\psi_0 + \pi_k(f_1) \psi_1 \]
where $f_0,f_1 \in C_c^\infty(\R)$ with supports contained in $(-\epsilon,\epsilon)$ and $\psi_0, \psi_1 \in C_c^\infty(M)$ with supports contained in $\mathrm{supp}(\varphi)$. Applying Theorem~\ref{linefac} with $X=X_{k-1}^M$ for $\psi =\psi_0$ and $\psi = \psi_1$ then gives
\[ \varphi = \pi_1(f_0) \pi_2(f_{00}) \psi_{00}  + \pi_1(f_0) \pi_2(f_{01}) \psi_{01} + \pi_1(f_1) \pi_2(f_{10}) \psi_{10} + \pi_1(f_1) \pi_2(f_{11}) \psi_{11} \]
where the $f$s are in $C_c^\infty(\R)$ with supports contained in $(-\epsilon, \epsilon)$ and $\psi$s are in $C_c^\infty(M)$ with supports  contained in $\mathrm{supp}(\varphi)$. Continuing in this manner, we eventually get $\varphi$ as the sum of $2^k$ terms of the form 
\[ \pi_1(f_1) \cdots \pi_k(f_k) \psi \]
where $f_1, \ldots, f_k \in C_c^\infty(\R)$ have supports contained in $(-\epsilon, \epsilon)$ and $\psi \in C_c^\infty(M)$ has support contained in $\mathrm{supp}(\varphi)$.  If $\epsilon$ is sufficiently small, then Lemma~\ref{horrid} guarantees that each of these terms can be written as $\pi(f)\psi$ where $f \in C_c^\infty(G)$ has $\mathrm{supp}(f) \subset W$.  
\end{proof}

Specializing to the case where $G$ is acting on itself, we obtain the desired generalized Dixmier-Malliavin theorem as a corollary.

\begin{thm}\label{mainprecis}
Let $G \rightrightarrows B$ be a Lie groupoid with a given Haar system. Then, for any $\varphi \in C_c^\infty(G)$, there exist $f_1,\ldots, f_N, \psi_1,\ldots, \psi_N \in C_c^\infty(G)$ such that 
\[ \varphi = f_1 * \psi_1 + \ldots + f_N * \psi_N. \]
Moreover, this factorization can be taken such that, for all $i$,  $\mathrm{supp}(\psi_i) \subset \mathrm{supp}(\varphi)$ and  $\mathrm{supp}(f_i) \subset W$, where $W$ is a prescribed open subset of $G$ containing $t(\mathrm{supp}(\varphi))$.
\end{thm}

\section{Product structure of ideals of smooth functions under multiplication}

In this section, we study ideals of  smooth functions  vanishing to given order along a submanifold when the operation is  pointwise multiplication. This will lay the groundwork for the subsequent section in which  the operation is convolution. The result we wish to generalize to the convolution setting is the following.

\begin{thm}\label{I2=I}
Let $X$ be a closed submanifold of a smooth manifold $M$. Let $I_p \subset C_c^\infty(M)$ denote the ideal of functions vanishing to $p$th order on $X$. Then, 
\begin{enumerate}
\item $(I_\infty)^2   = I_\infty$
\item $(I_1)^p = I_p$ for every positive integer $p$. 
\end{enumerate}
\end{thm}

The relation $(I_\infty)^2 = I_\infty$ actually  remains true even when  $X$ is any closed subset of $M$, and not necessarily a submanifold. This stronger result is due to Tougeron. See \cite{Tougeron}, Proposition~V.2.3 as well as  \cite{Thilliez}, Section~4. Note that, although the results in these references are stated in terms of germs of functions, it is a simple matter to use partitions of unity to convert them into statements about compactly-supported functions. The second relation $(I_1)^p = I_p$ is much more elementary than the first and can be established by applying Taylor's theorem locally. 

Theorem~\ref{I2=I} is not quite sufficient for our purposes, however. We need to consider a submersion  $\pi : N \to M$ (later taken to be the source or target projection of Lie groupoid) and the resulting $C_c^\infty(N)$-module structure on $C_c^\infty(M)$.  It will be convenient for us to combine the cases of infinite and finite vanishing order into a single statement, but we hasten to point out that the case of infinite vanishing order is by far the more substantive one.  Note also that the MathOverflow question \cite{MO} (still not fully resolved at  time of writing) centers around quite  similar issues. 



\begin{thm}\label{technicaltool}
Let $\pi : N \to M$ be a submersion. View $C_c^\infty(N)$ as a $C_c^\infty(M)$-module with product $f \cdot g = (f \circ \pi)g$, where $f \in C_c^\infty(M)$, $g \in C_c^\infty(N)$.  Let $X$ be a closed submanifold of $M$ and set $Y \colonequals \pi^{-1}(Y)$. For $p \in \N \cup \{ \infty\}$, write $I_p \subset C_c^\infty(M)$ and $J_p \subset C_c^\infty(M)$ for the ideals of functions that vanish to $p$th order on $X$ and $Y$ respectively. Then, the relation 
\[  J_{p + q} =  I_p \cdot J_q \] 
is satisfied for all $p,q \in \N \cup \{\infty\}$, where  $I_p \cdot J_q$ means the set of all sums of products $g \cdot h$, where $g\in I_p$, $h \in J_q$. 
\end{thm}

The problem is obviously local in nature; one can use a partition of unity to chop up a function $f$ on $N$ into smaller functions all of which vanish to the same order as $f$ on $Y$. In fact, it is enough to consider the case where $N = \R^k \times \R^\ell$, $M = \R^k$, $X = \{0\}$ and $\pi$ is the standard projection, so that $Y = \{0 \} \times \R^\ell$.  Throughout this section, $n = k+ \ell$ and  $\R^n = \R^k \times \R^\ell$ has coordinates $(x,y) = (x_1,\ldots,x_k, y_1,\ldots,y_\ell)$. We use the usual multi-index notation for partial derivatives: given $\gamma = (\alpha,\beta) \in \N^n = \N^k \times \N^\ell$ we  write $\partial^\gamma \colonequals \frac{\partial^{\alpha_1}}{\partial x_1^{\alpha_1}} \cdots  \frac{\partial^{\alpha_k}}{\partial x_k^{\alpha_k}}  \frac{\partial^{\beta_1}}{\partial y_1^{\beta_1}} \cdots \frac{\partial^{\beta_\ell}}{\partial y_\ell^{\beta_\ell}}$. 

We treat the  $p<\infty$ and $p=q=\infty$ cases of the local problem separately in the following two lemmas.  The bulk of our effort will go towards establishing the second of these.

\begin{lemma}\label{taylor}
If $f \in C^\infty(\R^n)$ vanishes to order $p+q$ on   $\{0\} \times \R^\ell$, where $p \in \N$, $q \in \N \cup \{\infty\}$, then one can write
\[ f(x,y) = \sum_{|\alpha| = p} x^\alpha f_\alpha(x,y), \]
where each $f_\alpha$ belongs to $C^\infty(\R^n)$ and vanishes to order $q$ on $\{0\} \times \R^\ell$. 
\end{lemma}

\begin{lemma}\label{inftaylor}
If $f \in C_c^\infty(\R^n)$ vanishes to order $\infty$ on $\{0\} \times \R^\ell$, then one can write 
\[ f(x,y) = \rho(x) h(x,y), \]
where $\rho \in C^\infty(\R^k)$ has a zero of order $\infty$ at $0$ and is strictly positive on $\R^k \setminus \{0\}$, and $h \in C^\infty(\R^n)$ vanishes to order $\infty$ on $\{0\} \times \R^\ell$. 
\end{lemma}

It is a simple matter to derive Theorem~\ref{technicaltool} from these lemmas.

\begin{proof}[Proof of Theorem~\ref{technicaltool}]
We just need to show $J_{p+q} \subset I_p \cdot J_q$, the reverse containment being obvious. Suppose, therefore, that $f \in J_{p+q}$. Using a partition of unity argument and standard facts about the local structure of submersions and submanifolds, we may assume one of the following two alternatives holds:
\begin{enumerate}[(i)]
\item The support of $f$ is disjoint from $Y$. 
\item The support of $f$ is contained in an open set $U$ such that:
\begin{enumerate}[(a)]
\item $U$ is diffeomorphic to $\R^k \times \R^{\ell_1} \times \R^{\ell_2}$ and  $\pi(U)$ is diffeomorphic to $\R^k \times \R^{\ell_1}$,  
\item under these diffeomorphisms, $\pi  : U \to \pi(U)$ identifies with the standard projection $\R^k \times \R^{\ell_1} \times \R^{\ell_2} \to \R^k \times \R^{\ell_1}$, and
\item under these diffeomorphisms, $X \cap \pi(U)$ identifies $\{0\} \times \R^{\ell_1}$, so that $Y \cap U$ identifies with $\{0\} \times \R^{\ell_1} \times \R^{\ell_2}$. 
\end{enumerate}
\end{enumerate}
If (i) is satisfied, take any $g \in C_c^\infty(M)$ that is equal to $1$ on $\pi( \mathrm{supp}(f))$ and equal to $0$ outside  of some open set  not intersecting $X$. Then,  $f=(g \circ \pi) f$, where $g$ vanishes to order $\infty \geq p$ on $X$ and $f$ vanishes to order $p+q \geq q$ on $Y$. 

If (ii) is satisfied and $p<\infty$, then,   applying Lemma~\ref{taylor} in the given chart with $\ell = \ell_1 + \ell_2$, we may assume that  $f = x^\alpha h$, where $|\alpha| = p$ and $h \in C_c^\infty(\R^k \times \R^{\ell_1} \times \R^{\ell_2})$ vanishes to order $q$ on $\{0\} \times \R^{\ell_1} \times \R^{\ell_2}$. Then, $f=(g \circ \pi) h$, where $g \in C_c^\infty(\R^k \times \R^{\ell_1})$ is given by $g=x^\alpha \varphi$ for an appropriate cutoff function $\varphi \in C_c^\infty(\R^k \times \R^{\ell_1})$. Obviously, $g$ vanishes to order $p$ on $\{0\} \times \R^{\ell_1}$. The expression $f =( g \circ \pi) h$ can be made global simply by extending $g$ and $h$ to be identically $0$ outside of $\pi(U)$ and $U$, respectively.

If (ii) is satisfied and $q=\infty$, we proceed in the same way using Lemma~\ref{inftaylor}. First, in the given chart, write $f(x,y) = \rho(x) h(x,y)$ where $\rho \in C_c^\infty(\R^k)$ has a zero of infinite order at $0$ and is strictly positive on $\R^k \setminus \{0\}$, and $h \in C^\infty_c(\R^k \times \R^{\ell_1} \times \R^{\ell_2})$ vanishes to order $\infty$ on $\{0\} \times \R^{\ell_1} \times \R^{\ell_2}$. Then, $f=(g\circ \pi)h$, where $g \in C_c^\infty(\R^k \times \R^{\ell_1})$ is given by $g = \rho \cdot  \varphi$ for an appropriate cutoff function $\varphi \in C_c^\infty(\R^k \times \R^{\ell_1})$. Obviously $g$ vanishes to order $\infty \geq p$ on $\{0\} \times \R^{\ell_1}$. 
\end{proof}

It remains to prove the Lemmas~\ref{taylor} and \ref{inftaylor}. 

\begin{proof}[Proof of Lemma~\ref{taylor}]
If $p=0$, there is nothing to prove, so assume $p  \geq 1$. It clearly suffices to prove that we can write
\[  f(x,y) = x_1 f_1(x,y) + \ldots x_k f_k(x,y) \]
where $f_i \in C^\infty(\R^n)$ vanish to order $p+q-1$ and proceed recursively. The functions $f_1,\ldots,f_k$ defined by
\begin{align*}
f_1(x,y) &= \int_0^1 \tfrac{\partial f}{\partial x_1}( tx_1,x_2,\ldots,x_k,y_1,y_2,\ldots,y_\ell)  \ dt \\
f_2(x,y) &= \int_0^1 \tfrac{\partial f}{\partial x_2}(  0,tx_2,x_3,\ldots, x_k,y_1,y_2,\ldots,y_\ell)  \ dt \\
\vdots &  \\
f_k(x,y) &= \int_0^1 \tfrac{\partial f}{\partial x_k}( 0, \ldots,0,t x_k,y_1,y_2,\ldots,y_\ell)  \ dt  
\end{align*}
serve this purpose.
\end{proof}

The proof of Lemma~\ref{inftaylor} will rely on several further lemmas, specifically Lemmas~\ref{decay0}, \ref{radialfn} and \ref{fonglemma}. Recall that a function $f$ on $[1,\infty)$ is  \emph{rapidly decaying} if $\lim_{t \to \infty} t^m f(t)=0$ for every nonnegative integer $m$. We say  $f$  is a \emph{Schwartz function} if it is $C^\infty$ and it and all its derivatives are rapidly decaying. Since there are many more functions of rapid decay than there are Schwartz functions, it seems plausible  that there could exist a  function of rapid decay that vanishes more slowly than any Schwartz function. The following lemma shows this does not occur by providing a ``Schwartz envelope'' for any rapidly decaying function. The original reference for this fact may be \cite{Chevalley}, Lemma~3.6, pp.~127. One can also find it in the expository note \cite{Garrett}.

\begin{lemma}\label{envelope}\leavevmode 
\begin{enumerate}
\item If $f$ is a bounded, rapidly decaying function on $[1,\infty)$, then there exists a positive-valued, monotone decreasing Schwartz function $g$ on $[1,\infty)$ such that $|f| \leq g$. 
\item If $(f_k)$ is a sequence of rapidly decaying functions on $[1,\infty)$, then there exists a positive-valued Schwartz function $g$ on $[1,\infty)$ such that $\lim_{t \to \infty} \frac{f_k(t)}{g(t)} = 0$ for all $k$.
\end{enumerate}
\end{lemma}
\begin{proof}For the first part, assume without loss of generality that  $f$ is monotone decreasing, or else replace it by $t \mapsto \sup_{s \geq t} |f(s)|$. Let $\varphi \in C^\infty(\R)$ be a nonnegative-valued function with support contained in $[0,1]$ satisfying $\int_\R \varphi(t) \ dt = 1$. Define $g$ to be the convolution $\varphi * f$, that is, $g(t) = \int_0^1 \varphi(s) f(t - s)  \ ds$. We remark that there is a small issue with this definition of $g$ near $t=1$, but this is easily fixed by enlarging the domain of $f$, say by defining $f(t) = \sup_{s \geq 1} f(s)$ for $t \leq 1$. It is easy to see that $g$ is monotone decreasing and that $f \leq g$. One can check that the convolution of two rapidly decaying functions is rapidly decaying (imposing sufficient regularity properties so that convolution makes sense), and it follows that the convolution of a rapidly decaying function with a Schwartz function is Schwartz (since the derivatives can be put on the Schwartz function). 

For the second part, assume without loss of generality that each $f_k$ is bounded and use the first part to produce, for each $k$, a positive-valued Schwartz function $g_k$ such that $\lim_{t \to \infty} \frac{f_k(t)}{g_k(t)} = 0$ (if  $f_k \leq g_k$ holds, but $\frac{f_k}{g_k}$ does not vanish at infinity, replace $g_k$  with  $t \mapsto t g_k(t)$). An easy diagonal selection argument guarantees the existence of constants $c_k > 0$ such that $g = \sum c_k g_k$ is a Schwartz function. Since $g > g_k$, it is clear that $\frac{f_k}{g}$ vanishes at infinity for every $k$. 
\end{proof}
The next step is to convert Lemma~\ref{envelope} into a statement about smooth functions with an infinite order zero at $0$ by performing an inversion in the variable. Much more sophisticated accounts of the connection between Schwartz functions and functions that remain smooth after being extended by zero can be found in the literature, see \cite{Aizenbud-Gourevitch}, Theorem 5.4.1. For present purposes, the simple-minded lemma below is  enough.

 \begin{lemma}\label{corresp}
The inversion map $t \mapsto 1/t: (0,1] \to [1,\infty)$ puts functions $f$ on $(0,1]$ with $\lim_{t \to 0^+} f(t) t^{-m} = 0$ for all positive integers $m$ into bijection with the rapidly decaying functions on $[1,\infty)$, and also puts the smooth functions $f$ on $(0,1]$ for which putting $f(t) = 0$ for $t \leq 0$ yields a smooth extension into bijection with the Schwartz functions on $[1,\infty)$. 
\end{lemma}

To prove Lemma~\ref{corresp}, we need the following simple fact. 

\begin{lemma}\label{smooth1D}
Suppose $f$ is a smooth function on $(0,\infty)$ with $\lim_{t \to 0^+} f^{(m)}(t) = 0$ for every nonnegative integer $m$. Then, in setting $f(t) = 0$ for $t \leq 0$, one obtains a $C^\infty$ extension of $f$ to all of $\R$.
\end{lemma}
\begin{proof}
An application of the mean value theorem shows the extension is differentiable with derivative $0$  at the origin. The statement follows by induction.
\end{proof}

\begin{proof}[Proof of Lemma~\ref{corresp}]
The first correspondence is obvious. Towards the second, suppose $f$ is a Schwartz function on $[1,\infty)$ and define $g$ on $(0,1]$ by $g(t) = f(1/t)$. Then, $g'(t) = f_1(1/t)$, where $f_1(t) = - t^2 f'(t)$ is yet another Schwartz function. By induction, each derivative of $g$ has the form $f_k(1/t)$ for some Schwartz function $f_k$ on $[1,\infty)$. In particular, $\lim_{t \to 0^+} g^{(k)}(t) = 0$ for all $k$ so that, by Lemma~\ref{smooth1D}, setting $g(t) = 0$ for $t \leq 0$ effects a smooth extension of $g$. The converse direction, that $f(t) = g(1/t)$ is a Schwartz function on $[1,\infty)$ when $g$ is a smooth function with $g(t)=0$ for $t \leq 0$, proceeds similarly.  
\end{proof}

Applying the correspondence of Lemma~\ref{corresp}, the second part of Lemma~\ref{envelope} translates to the following.

\begin{lemma}\label{decay0}
Let $f_k$ be a sequence of functions on $[0,\infty)$ that vanish to infinite order at $0$, i.e. $f_k(0) = 0$ and $\lim_{t \to 0^+} f_k(t) t^{-m} = 0$ for all $m$. Then, there exists a $C^\infty$ function $g$ on $\R$ with $g(t) = 0$ for $t \leq 0$, and $g(t) > 0$ for $t>0$ such that $\lim_{t \to 0^+} \frac{f_k(t)}{g(t)} =0$ for all $k$. 
\end{lemma}

\begin{rmk}
A direct proof of Lemma~\ref{decay0} was given by George Lowther at \cite{MO} (Lemma 2). Nonetheless, as the Schwartz function formulation of this result appears to be   better known, it seemed worthwhile to draw out this connection here.
\end{rmk}

Lemmas~\ref{radialfn} and \ref{fonglemma}, stated and proved below, will rely on the following mild generalization of  Lemma~\ref{smooth1D} whose proof we omit. Recall that $n = k+ \ell$ and  $\R^n = \R^k \times \R^\ell$ has coordinates $(x,y) = (x_1,\ldots,x_k, y_1,\ldots,y_\ell)$.

\begin{lemma}\label{smoothnd}
Let $f$ be a smooth function on $\R^n \setminus (\{0\} \times \R^\ell)$ such that, for every $\gamma \in \N^n$, the  partial derivative  $\partial^\gamma f$ has limit zero at every point of $\{0\} \times \R^\ell$. Then, $f$ extends to a $C^\infty$ function on all of $\R^n$ vanishing to infinite order on $\{0\} \times \R^\ell$. 
\end{lemma}

In particular, when $\ell=0$, the above says that a smooth function on $\R^k \setminus \{0\}$ with all higher partials vanishing at the origin extends smoothly to all of  $\R^k$.  This is helpful in checking the following.

\begin{lemma}\label{radialfn}
Let $\varphi$ be a smooth function on $\R$ with a  zero of infinite order at $0$. Then, the function  $f$ on $\R^k$ defined by $f(x)=\varphi(|x|)$, where $|x| = \sqrt{x_1^2 + \ldots + x_k^2}$, is a $C^\infty$ function on $\R^k$ with a zero of infinite order at $0$. 
\end{lemma}
\begin{proof}
Obviously $f$ is smooth on $\R^k \setminus \{0\}$. On the latter domain, $\frac{\partial f}{\partial x_i}(x) = \psi(|x|)$ where $\psi(t)=\begin{cases} - \frac{\varphi(t)}{t} & t \neq 0 \\ 0 & t=0\end{cases}$ is another $C^\infty$ function on $\R$ with a zero of infinite order at $0$. By induction, $f$ satisfies the conditions of Lemma~\ref{smoothnd} (with $\ell =0$), whence is smooth as claimed.
\end{proof}

The next lemma gives sufficient conditions under which the quotient of two smooth functions on $\R^n$ that vanish to infinite order on $\{0\} \times \R^\ell$ is another such function.

\begin{lemma}\label{fonglemma}
Let $f$ and $g$ be  $C^\infty$ functions on $\R^n$ that vanish to infinite order on $\{0\} \times \R^\ell$  and assume $g >0$  on $\R^n \setminus ( \{0\} \times \R^\ell)$. If, for every $\gamma \in \N^n$ and $m \in \N$, the function $\frac{\partial^\gamma f}{g^m}$ has limit $0$ at each point of $\{0\} \times \R^\ell$, then $\frac{f}{g}$ extends to a $C^\infty$ function on all of $\R^n$ vanishing to infinite order on $\{0\} \times \R^\ell$. 
\end{lemma}

\begin{proof}
Let $\mathscr{F}$ denote the collection of all smooth functions on $\R^n \setminus (\{0\} \times \R^\ell)$ obtained as $C^\infty(\R^n)$-linear combinations of the functions $\frac{\partial^\gamma f}{g^m}$. By assumption, the functions in $\mathscr{F}$ all have limit $0$ at each point $\{0\} \times \R^\ell$. Observe that $\mathscr{F}$ is closed under taking partial derivatives. Indeed, if $\gamma$ is a multi-index, $m$ is a positive integer, $h \in C^\infty(\R^n)$ and $\partial$ is one of the first-order partials $\frac{\partial}{\partial x_i}$ or $\frac{\partial}{\partial y_j}$, then $\partial ( h \frac{\partial^\gamma f}{g^m} )= (\partial h) \frac{\partial^\gamma f}{g^m} + h \frac{\partial \circ \partial^\gamma f}{g^m} - m(\partial g)h \frac{\partial^\gamma f}{g^{m+1}}$. Thus, thinking of $\frac{f}{g}$ as a smooth function on $\R^n \setminus (\{0\} \times \R^\ell)$,  we have by induction that all of its higher order partial derivatives have limit $0$ at every point of $\{0\} \times \R^\ell$ and so, by  Lemma~\ref{smoothnd}, $\frac{f}{g}$ extends to a smooth function on $\R^n$ that vanishes to infinite order on $\{0\} \times \R^\ell$. 
 \end{proof}

We are now in a position to give the

\begin{proof}[Proof of Lemma~\ref{inftaylor}]
Suppose that $f \in C_c^\infty(\R^n)$ vanishes to order $\infty$ on $\{0\} \times \R^\ell$. Given  $\gamma \in \N^n$ and $m,r \in \N$, define a continuous function $f_{\gamma,m,r}$ on $[0,\infty)$ by
\[ f_{\gamma,m,r}(t) = \sup_{\substack{\ x \in \R^k \\  |x| \leq t }} \ \sup_{\substack{\ y \in \R^\ell \\  |y| \leq r}} | \partial^\gamma f(x,y)|^{1/m}. \]
The assumption that $f$ vanishes to infinite order on $\{0\} \times \R^\ell$ implies that each $f_{\gamma, m,r}$ vanishes 
  to infinite order at $t=0$, i.e. $\lim_{t \to 0^+} f_{\gamma,m,r}(t) t^{-s} = 0$ for any positive integer $s$. It therefore follows from Lemma~\ref{decay0} that there exists a $C^\infty$ function $\varphi$ on $\R$ vanishing to infinite order at $t=0$ with $\varphi(t)>0$ for $t>0$ such that $\lim_{t \to 0^+} \frac{ f_{\gamma,m,r}(t)}{\varphi(t)} =0$ for all $\gamma,m,r$. By Lemma~\ref{radialfn}, the function  $\rho$ on $\R^k$  defined by $\rho(x)=\varphi(|x|)$ is a $C^\infty$ function, positive on $\R^k \setminus \{0\}$ and vanishing to infinite order at $0$. By design, for any $\gamma \in \N^n$ and $m,r \in \N$, one has the bound
\[ \left| \frac{ \partial^\gamma f (x,y)}{\rho(x)^m} \right| \leq
\left( \frac{f_{\gamma, m,r}(|x|)}{\varphi(|x|)} \right)^m \]
for $x \neq 0$ and $|y| < r$, which shows that the left hand side vanishes as $x \to 0$. Thus, applying Lemma~\ref{fonglemma}, one has that $h(x,y) =  \frac{f(x,y)}{\rho(x)}$ extends smoothly to a function on all of $\R^n$ that vanishes to infinite order on $\{0\} \times \R^\ell$, completing the proof.
\end{proof}

\section{Product structure of ideals in the smooth convolution algebra of a Lie groupoid}




In this final section, we apply  our generalization of the Dixmier-Malliavin theorem to obtain the Lie groupoid analog of  Theorem~\ref{I2=I} by reducing it to its commutative counterpart. Throughout,  $G$ denotes  a Lie groupoid over the manifold $M$ with fixed Haar system $\lambda$ and we assume that $X \subset M$ is an  \emph{invariant} closed submanifold in the sense that $s^{-1}(X)=t^{-1}(X)$. The restriction $G_X \colonequals s^{-1}(X) = t^{-1}(X)$ of $G$ to $X$ is, in its own right, a Lie groupoid $G_X \rightrightarrows X$. The Haar system $\lambda$ on $G$ can be restricted to a Haar system $\lambda_X$ on $G_X$ and doing so makes the restriction map $C_c^\infty(G) \to C_c^\infty(G_X)$ into a homomorphism of the smooth convolution algebras. The kernel of this homomorphism is the ideal $J_1 \subset C_c^\infty(G)$ of functions that vanish on $G_X$. More generally, 
one can consider $J_p \subset C_c^\infty(G)$, the functions which vanish to $p$th order on $G_X$. It is simple to confirm that each $J_p$ is an ideal with respect to the convolution product (either by arguing directly, or by applying Proposition~\ref{pullout} below). The quotients $C_c^\infty(G) / J_p$ for $p > 1$ can be thought of as extensions  of the convolution algebra $C_c^\infty(G_X)$, fitting as they do into exact sequences of the form 
\[ 0 \to J_1 / J_p \to C_c^\infty(G) / J_p \to C_c^\infty(G_X) \to 0. \]
Roughly speaking, the kernel $J_1/J_p$ contains Taylor series information up to order $p-1$ in directions transverse to $G_X$. 

 The Lie groupoid algebra analog of Theorem~\ref{I2=I} is the following:

\begin{thm}\label{J2=J}
Let $G \rightrightarrows M$ be a Lie groupoid with given Haar system. Let $X$ be an invariant, closed submanifold of $M$ and let $G_X \colonequals s^{-1}(X) = t^{-1}(X)$. Let $J_p  \subset  C_c^\infty(G)$ denote the ideal, with respect to convolution, of functions that vanish to order $p$ on $G_X$. Then,  
\begin{enumerate}
\item $J_\infty * J_\infty  = J_\infty$
\item $(J_1)^{*p} = J_p$ for every positive integer $p$. 
\end{enumerate}
\end{thm}



As in the preceding section, for the sake of efficiency, we shall in fact prove a more general result which treats the cases of finite and infinite vanishing order on equal footing.

\begin{thm}\label{applthm}
Let $G \rightrightarrows M$ be a Lie groupoid with a given Haar system. Let $X$ be an invariant, closed submanifold of $M$ and let $G_X \colonequals s^{-1}(X) = t^{-1}(X)$. For $p \in \N \cup \{\infty\}$, let $J_p  \subset  C_c^\infty(G)$ denote the ideal of functions that vanish to order $p$ on $G_X$. Then,  
\[J_{p+q} =  J_p*J_q \] 
holds for all $p,q \in \N \cup \{\infty\}$. 
\end{thm}


It is easy to see that  $J_p * J_q \subset J_{p + q}$ is satisfied (again, either by arguing directly  or by applying   Proposition~\ref{pullout} below). The goal is therefore to sharpen these containments to equalities.  Note that, whereas in the commutative setting the  $p=q=0$ is trivial, in  Theorem~\ref{applthm} above the $p=q=0$ case is exactly Theorem~\ref{mainprecis},  our extension of the Dixmier-Malliavin theorem . Conversely, Theorem~\ref{mainprecis}, in tandem with Proposition~\ref{pullout} below, reduces the proof of Theorem~\ref{applthm} to a formal manipulation. 

Recall that  $C_c^\infty(G)$ is a $C_c^\infty(M)$-bimodule with respect to the products defined by
\begin{align*}
f \cdot \varphi = (f \circ t) \varphi && \varphi \cdot f = \varphi (f \circ s)
\end{align*}
and, moreover, that these products satisfy the expected associativity identities
\begin{align*}
f \cdot (\varphi * \psi) = (f \cdot \varphi) * \psi && (\varphi * \psi) \cdot f = \varphi *(\psi \cdot f),
\end{align*}
where $f \in C_c^\infty(M)$ and $\varphi, \psi \in C_c^\infty(G)$.

The following proposition shows that the ideals $I_p \subset C_c^\infty(M)$ of functions vanishing to $p$th order on $X$ determine the ideals $J_p \subset C_c^\infty(G)$ of functions vanishing to $p$th order on $G_X$ by way of this module structure;  one may write $J_p = I_p \cdot C_c^\infty(G) = C_c^\infty(G) \cdot I_p$. It is a quick corollary of the results in the preceding section.
\begin{propn}\label{pullout}
Let $G \rightrightarrows M$ be Lie groupoid with a given Haar system. Let $X$ be an invariant closed submanifold of $M$. For each $p \in \N \cup \{\infty\}$, let $I_p \subset C_c^\infty(M)$ and $J_p \subset C_c^\infty(G)$ denote the collection of functions vanishing to $p$th order on $X$ and $G_X$ respectively. Then,  
\[ J_{p+q} = I_p \cdot J_q= J_q \cdot I_p \]
holds for all $p,q \in \N \cup \{\infty\}$.
\end{propn}
\begin{proof}
Apply Theorem~\ref{technicaltool} with  $N=G$ and $\pi = s$, respectively $\pi = t$. 
\end{proof}

Theorem~\ref{applthm} is now a trivial consequence of  Theorem~\ref{mainprecis} and Proposition~\ref{pullout}.

\begin{proof}[Proof of Theorem~\ref{applthm}]
We have
\[ J_p * J_q = I_p \cdot C_c^\infty(G) * C_c^\infty(G) \cdot I_q =I_p \cdot C_c^\infty(G) \cdot I_q = J_p \cdot I_q = J_{p+q}, \]
where the second equality holds by Theorem~\ref{mainprecis} and the rest  hold by Proposition~\ref{pullout}.
\end{proof}

\bibliographystyle{abbrv}
\bibliography{DMBib}

\end{document}